\documentclass[a4paper,twoside,reqno,11pt,tbtags]{article}

\usepackage{amssymb,amsmath,amsthm,mathrsfs}
\usepackage{graphicx}
\usepackage{times}
\usepackage[round]{natbib}

\leftmargini=\baselineskip

\newtheoremstyle{localthm}
	{5pt} % space above
	{5pt} % space below
	{\sl} % Body font
	{} % Indent amount
	{\bf} % Theorem head font
	{{\rm.}} % Punctuation after theorem head
	{.7em} % Space after theorem head
	{} % Theorem head spec ?

\theoremstyle{localthm}
\newtheorem{Theorem}{Theorem}%[section]

\newtheorem{Proposition}[Theorem]{Proposition}

\newtheoremstyle{localrem}
	{5pt} % space above
	{5pt} % space below
	{\rm} % Body font
	{} % Indent amount
	{\bf} % Theorem head font
	{{\rm.}} % Punctuation after theorem head
	{.7em} % Space after theorem head
	{} % Theorem head spec ?

\theoremstyle{localrem}

\def\Pr{\mathop{\mathrm{I\!P}}\nolimits}
\newcommand{\Ex}{\mathop{\mathrm{I\!E}}\nolimits}

\newcommand{\Cov}{\mathop{\mathrm{Cov}}\nolimits}
\newcommand{\Var}{\mathop{\mathrm{Var}}\nolimits}

\newcommand{\R}{\mathbb{R}}

\newcommand{\EE}{\mathcal{E}}
\newcommand{\LL}{\mathcal{L}}
\newcommand{\NN}{\mathcal{N}}
\newcommand{\PP}{\mathcal{P}}
\newcommand{\QQ}{\mathcal{Q}}
\def\SS{\mathcal{S}}

\newcommand{\diag}{\mathop{\mathrm{diag}}\nolimits}
\newcommand{\tr}{\mathop{\mathrm{tr}}\nolimits}

\newcommand{\bs}{\boldsymbol}

\newcommand{\bSigma}{\bs{\Sigma}}

\newcommand{\Rqq}{\R_{}^{q\times q}}

\newcommand{\Rqqsym}{\R_{\rm sym}^{q\times q}}
\newcommand{\Rqqsympd}{\R_{{\rm sym},+}^{q\times q}}

\usepackage{color}

\begin{document}
%%%%%%%%%%%%%%%%

\title{Approximating Symmetrized Estimators of Scatter\\
	via Balanced Incomplete $U$-Statistics}
\author{Lutz D\"umbgen$^{*}$ and Klaus Nordhausen\\
	University of Bern and University of Jyv\"askyl\"a}
\date{}
\maketitle

\begin{abstract}
We derive limiting distributions of symmetrized est\-im\-ators of scatter, where instead of all $n(n-1)/2$ pairs of the $n$ observations we only consider $nd$ suitably chosen pairs, $1 \le d < \lfloor n/2\rfloor$. It turns out that the resulting estimators are asymptotically equivalent to the original one whenever $d = d(n) \to \infty$ at arbitrarily slow speed. We also investigate the asymptotic properties for arbitrary fixed $d$. These considerations and numerical examples indicate that for practical purposes, moderate fixed values of $d$ between $10$ and $20$ yield already estimators which are computationally feasible and rather close to the original ones.
\end{abstract}

\vfill

\noindent
$^*$Work supported by Swiss National Science Foundation.\\

\paragraph{AMS subject classifications:}
62H12, 65C60.

\paragraph{Key words:}
Asymptotic normality, Incomplete $U$-statistic, independent component analysis, linear expansion, $U$-statistic.

\paragraph{Corresponding author:}
Lutz D\"umbgen, e-mail: {\tt duembgen@stat.unibe.ch}

\addtolength{\baselineskip}{0.2\baselineskip}

%=======================
\section{Introduction}
\label{sec:introduction}
%=======================

Robust estimation of multivariate scatter for a distribution $P$ on $\R^q$, $q \ge 1$, is a recurring topic in statistics. For instance, different estimators of multivariate scatter are an important ingredient for independent component analysis (ICA) or invariant coordinate selection (ICS), see \cite{Nordhausen_etal_2008}, \cite{Tyler_etal_2009} and the references therein. Other potential applications are classification methods and multivariate regression, see for instance \cite{Nordhausen_Tyler_2015}. Of particular interest are symmetrized estimators of scatter which are defined in Section~\ref{sec:scatter.functionals}. Throughout this paper we consider independent random vectors $X_1, X_2, \ldots, X_n$ with distribution $P$. The symmetrized estimators are just standard functionals of scatter (with given center $0 \in \R^q$) applied to the empirical distribution
\[
	\hat{Q}_n \
	:= \ \binom{n}{2}^{-1} \sum_{1 \le i < j \le n} \delta_{X_j - X_i}^{\rm s} ,
\]
where $\delta_z^{\rm s} := 2^{-1} (\delta_z + \delta_{-z})$, and $\delta_z$ denotes Dirac measure at $z \in \R^q$. Thus, $\hat{Q}_n$ is the empirical distribution of all $n(n-1)$ differences of two different observations, and it may be viewed as a measure-valued version of a $U$-statistic as introduced by \cite{Hoeffding_1948}. It is an unbiased estimator of the symmetrized distribution
\begin{equation}
\label{eq:symmetrized.P}
	Q = Q(P) \ := \ \LL(X_1 - X_2) .
\end{equation}
Here and throughout, $\LL(\cdot)$ stands for `distribution of'. The computation of symmetrized $M$-estimators of scatter is rather time-consuming, whence some people refrain from using them. However, the symmetrized estimators have two desirable properties: one avoids the estimation of a location nuisance parameter, and the underlying scatter functional has the so-called block independence property as explained in Section~\ref{sec:scatter.functionals}; see also \cite{Duembgen_1998} and \cite{Sirkiae_etal_2007}.

To diminish the computational burden, one could replace the empirical distribution $\hat{Q}_n$ with the empirical distribution
\[
	\hat{Q}_{n,d} \
	:= \ (nd)^{-1} \sum_{i=1}^n \sum_{j=i+1}^{i+d}
		\delta_{X_j - X_i}^{\rm s}
\]
for some integer $1 \le d \le (n-1)/2$, where $X_{n+s} := X_s$ for $1 \le s \le n$. This is a measure-valued version of a reduced $U$-statistic as introduced by \cite{Blom_1976} and \cite{Brown_Kildea_1978}. Other authors, e.g.\ \cite{Lee_1990}, call this a balanced incomplete $U$-statistic. In the context of estimation of scatter, \cite{Miettinen_etal_2016} illustrate the potential benefits of $\hat{Q}_{n,d}$ compared to $\hat{Q}_n$ in simulations. As a preliminary proof of concept, they present the asymptotic properties of the estimator $2^{-1} \int_{\R^q} yy^\top \, \hat{Q}_{n,d}(dy)$ in comparison to the usual sample covariance matrix $2^{-1} \int_{\R^q} yy^\top \hat{Q}_n(dy)$. Their findings are encouraging, but the latter estimator can be computed rather easily in $O(n)$ steps and is non-robust of course.

The purpose of the present paper is to provide an in-depth analysis of robust and smooth symmetrized scatter estimators based on $\hat{Q}_n$ and $\hat{Q}_{n,d}$, where the computation time with $\hat{Q}_n$ can definitely become a limiting factor. It turns out that these two scatter estimators are asymptotically equivalent whenever $d = d(n) \to \infty$. Here and throughout the sequel, asymptotic statements are meant as $n \to \infty$. More precisely, if $\bSigma(\cdot)$ is our functional of scatter, then it will be shown that the following statements are true: There exist two stochastically independent and centered Gaussian random matrices $\bs{G}_1, \bs{G}_2$ whose distribution depends only on $P$ and $\bSigma(\cdot)$ such that
\[
	\sqrt{n} \bigl( \bSigma(\hat{Q}_{n,d(n)}) - \bSigma(Q) \bigr)
	\ = \ \sqrt{n} \bigl( \bSigma(\hat{Q}_n) - \bSigma(Q) \bigr) + o_p(1)
	\ \to_{\LL} \ \bs{G}_1
\]
provided that $d(n) \to \infty$. Here `$o_p(1)$' denotes a random term converging to zero in probability, and `$\to_{\LL}$' denotes convergence in distribution. For any fixed integer $d \ge 1$,
\[
	\sqrt{n} \bigl( \bSigma(\hat{Q}_{n,d}) - \bSigma(Q) \bigr)
	\ \to_{\LL} \ \bs{G}_1 + d^{-1/2} \bs{G}_2 .
\]
This explains why for sufficiently large but fixed $d$, the estimator $\bSigma(\hat{Q}_{n,d})$ is a good surrogate for $\bSigma(\hat{Q}_n)$.

An easy way to compute $\bSigma(\hat{Q}_{n,d})$ is to generate a data matrix containing the $nd$ differences $X_j - X_i$, where $i \in \{1,\ldots,n\}$ and $j \in \{i+1,\ldots,i+d\}$, and to apply $\bSigma(\cdot)$ to the empirical distribution $(nd)^{-1} \sum_{k=1}^{nd} \delta_{Y_k}^{\rm s}$ of these $nd$ vectors $Y_k$. But for large values of $nd$, this may be too cumbersome. A possible alternative is to compute the average $d^{-1} \sum_{\ell=1}^d \bSigma(\hat{Q}_{n,1}^{(\ell)})$ with the same $d \ge 1$, where $\hat{Q}_{n,1}^{(1)}, \ldots, \hat{Q}_{n,1}^{(d)}$ are defined as $\hat{Q}_{n,1}$, but with $d$ random permutations of the observations $X_1,X_2, \ldots,X_n$. It turns out that for fixed $d$, this average has the same asymptotic distribution as $\bSigma(\hat{Q}_{n,d})$.

The remainder of this paper is organized as follows: In Section~\ref{sec:scatter.functionals}, we recall some basic facts about scatter functionals and symmetrized scatter functionals as presented by \cite{Duembgen_etal_2015}. In Section~\ref{sec:asymptotics}, the asymptotic results mentioned before are stated in detail. The theory is illustrated with numerical examples in Section~\ref{sec:examples}. All proofs are deferred to Section~\ref{sec:proofs} and Appendix~A. The starting point is standard theory for complete and incomplete $U$-statistics as presented, for instance, by \cite{Serfling_1980} and \cite{Lee_1990}. Suitable modifications of these results, combined with linear expansions for functionals of scatter yield the asymptotic distributions of $\bSigma(\hat{Q}_n)$ and $\bSigma(\hat{Q}_{n,d})$. For the averaging estimator $d^{-1} \sum_{\ell=1}^d \bSigma(\hat{Q}_{n,1}^{(\ell)})$, we derive and use a variation of the combinatorial central limit theorem of \cite{Hoeffding_1951}. This result is potentially of independent interest, for instance, in the context of kernel mean embeddings as used in machine learning \citep{Muandet_etal_2020}.

%======================================
\section{Functionals of Scatter}
\label{sec:scatter.functionals}
%======================================

The material in this section is adapted from the survey of \cite{Duembgen_etal_2015}, where the latter builds on previous work of \cite{Tyler_1987a}, \cite{Kent_Tyler_1991} and \cite{Dudley_etal_2009}.

The space of symmetric matrices in $\R^{q\times q}$ is denoted by $\Rqqsym$, and $\Rqqsympd$ stands for its subset of positive definite matrices. The identity matrix in $\Rqq$ is written as $I_q$. The Euclidean norm of a vector $v \in \R^q$ is denoted by $\|v\| = \sqrt{v^\top v}$. For matrices $M, N$ with identical dimensions we write
\[
	\langle M, N\rangle \ := \ \tr(M^\top N)
	\quad\text{and}\quad
	\|M\| \ := \ \sqrt{\langle M, M\rangle} ,
\]
so $\|M\|$ is the Frobenius norm of $M$.

%-------------------------------------------------------------
\subsection{Functionals of scatter for centered distributions}
%-------------------------------------------------------------

Let $\QQ$ be a given family of probability distributions on $\R^q$ which are viewed as centered around $0$. In our specific applications, this is plausible, because $\QQ$ consists of symmetrized distributions. We consider a function $\bSigma : \QQ \to \Rqqsympd$, called a functional of scatter, and $\bSigma(Q)$ is the scatter matrix of $Q \in \QQ$. For the general theory presented in the next section, we assume that $\QQ$ and $\bSigma$ have two important properties.

\paragraph{Linear equivariance.}
We assume that for any nonsingular matrix $B \in \Rqq$ and any distribution $Q \in \QQ$, the distribution $Q^B := \LL(BY)$ with $Y \sim Q$ belongs to $\QQ$ too, and that
\begin{equation}
\label{eq:equivariance}
	\bSigma(Q^B) \ = \ B \bSigma(Q) B^\top .
\end{equation}

Linear equivariance has some interesting implications. For instance, if $Q \in \QQ$ is spherically symmetric in the sense that $Q^B = Q$ for all orthogonal matrices $B \in \Rqq$, then $\bSigma = c I_q$ for some $c > 0$. Furthermore, if $Q^B = Q$ for some matrix $B = \diag(\xi_1,\ldots,\xi_q)$ with $\xi \in \{-1,1\}^q$, then for arbitrary different indices $i,j \in \{1,\ldots,q\}$, the $(i,j)$-th component of $\bSigma(Q)$ satisfies
\begin{equation}
\label{eq:sign.symmetry}
	\bSigma(Q)_{ij} \ = \ 0 \quad\text{whenever} \ \xi_i \ne \xi_j .
\end{equation}

\paragraph{Differentiability.}
We assume that $\QQ$ is an open subset of the family of all probability distributions on $\R^q$ in the topology of weak convergence. Moreover, for any distribution $Q \in \QQ$, there exists a bounded, measurable and even function $J = J_Q : \R^q \to \Rqqsym$ such that $\int_{\R^q} J \, dQ = 0$, and for other distributions $\check{Q} \in \QQ$,
\[
	\bSigma(\check{Q}) \ = \ \bSigma(Q) + \int_{\R^q} J \, d\check{Q}
		+ o \Bigl( \Bigl\| \int_{\R^q} J \, d\check{Q} \Bigr\| \Bigr)
\]
as $\check{Q} \to Q$ weakly. Note that this differentiability property of $\bSigma(\cdot)$ implies its robustness in the sense that $\bSigma(\check{Q}) \to \bSigma(Q)$ as $\check{Q} \to Q$ weakly, because then $\int_{\R^q} J \, d\check{Q} \to \int_{\R^q} J \, dQ = 0$.

\paragraph{$M$-functionals of scatter.}
An important example for $\bSigma$ are $M$-functionals of scatter, driven by a function $\rho : [0,\infty) \to \R$ with the following properties: $\rho$ is twice continuously differentiable such that $\psi(s) := s \rho'(s)$ satisfies the inequalities $\psi'(s) > 0$ for $s > 0$ and $q < \psi(\infty) := \lim_{s \to \infty} \psi(s) < \infty$. For any distribution $Q$ on $\R^q$ and $\Sigma \in \Rqqsympd$, let
\begin{equation}
\label{eq:L.rho}
	L_\rho(\Sigma,Q)
	\ := \ \int_{\R^q} \bigl[ \rho(y^\top \Sigma^{-1}y) - \rho(y^\top y) \bigr] \, Q(dy)
		+ \log \det(\Sigma) .
\end{equation}
The function $L_\rho(\cdot,Q)$ has a unique minimizer $\bSigma(Q)$ on $\Rqqsympd$ if and only if
\[
	Q(\mathbb{W}) \ < \ \frac{\psi(\infty) - q + \dim(\mathbb{W})}{\psi(\infty)}
\]
for any linear subspace $\mathbb{W}$ of $\R^q$ with $0 \le \dim(\mathbb{W}) < q$. The set $\QQ$ of distributions which satisfy the latter constraints is open with respect to weak convergence.

A particular example for a function $\rho$ with the stated properties is given by $\rho(s) = \rho_\nu(s) := (\nu + q) \log(s + \nu)$, where $\nu > 0$.

The function $J = J_Q$ is rather complicated in general. But in case of a spherically symmetric distribution $Q$ with $\bSigma(Q) = I_q$,
\[
	J(y) \
	= \  \frac{q+2}{q+2+2\kappa} \, \rho'(\|y\|^2) \Bigl( yy^\top - \frac{\|y\|^2}{q} I_q \Bigr)
		+ \frac{1}{1 + \kappa} \Bigl( \rho'(\|y\|^2) \frac{\|y\|^2}{q} - 1 \Bigr) I_q
\]
for $y \in \R^q$, where $\kappa := q^{-1} \int_{\R^q} \rho''(\|y\|^2) \|y\|^4 \, Q(dy) \in (-1,\infty)$.

%------------------------------------------------
\subsection{Tyler's (1987) functional of scatter}
%------------------------------------------------

For any distribution $Q$ on $\R^q$ such that $Q(\{0\}) = 0$ and $\Sigma \in \Rqqsympd$, let
\[
	L_0(\Sigma,Q)
	\ := \ q \int_{\R^q} \log \Bigl( \frac{y^\top \Sigma^{-1}y}{y^\top y} \Bigr) \, Q(dy)
		+ \log \det(\Sigma) .
\]
Note that $L_0(t\Sigma,Q) = L_0(\Sigma,Q)$ for all $t > 0$. The function $L_0(\cdot,Q)$ has a unique minimizer $\bSigma_0(Q)$ on the set $\bigl\{ \Sigma \in \Rqqsympd : \det(\Sigma) = 1 \bigr\}$ if and only if
\[
	Q(\mathbb{W}) \ < \ \frac{\dim(\mathbb{W})}{q}
\]
for any linear subspace $\mathbb{W}$ of $\R^q$ with $1 \le \dim(\mathbb{W}) < q$. The set of all distributions $Q$ which satisfy the latter constraints and $Q(\{0\}) = 0$ is denoted by $\QQ_0$.

The functional $\bSigma_0$ satisfies a restricted equivariance property: For any $Q \in \QQ_0$ and any nonsingular matrix $B \in \Rqq$ with $|\det(B)| = 1$, equation~\eqref{eq:equivariance} holds true with $\bSigma_0$ in place of $\bSigma$. This implies that $\bSigma_0(Q) = I_q$ if $Q$ is spherically symmetric. Moreover, if $Q^B = Q$ with $B = \diag(\xi_1,\ldots,\xi_q)$ and $\xi \in \{-1,1\}^q$, then \eqref{eq:sign.symmetry} is satisfied with $\bSigma_0$ in place of $\bSigma$.

The functional $\bSigma_0$ is also differentiable in the following sense: For any distribution $Q \in \QQ_0$ there exists a bounded, continuous and even function $J : \R^q \setminus \{0\} \to \Rqqsym$ such that $\int_{\R^q} J \, dQ = 0$, $\mathrm{trace} \bigl( \bSigma_0(Q)^{-1} J \bigr) \equiv 0$, and for any distribution $\check{Q} \in \QQ$,
\[
	\bSigma_0(\check{Q}) \ = \ \bSigma_0(Q) + \int_{\R^q} J \, d\check{Q}
		+ o \Bigl( \Bigl\| \int_{\R^q} J \, d\check{Q} \Bigr\| \Bigr)
\]
as $\check{Q} \to Q$ weakly. Again, the function $J = J_Q$ is rather complicated in general. But in case of a spherically symmetric distribution $Q \in \QQ_0$,
\[
	J(y) \ = \ (q+2) \bigl( \|y\|^{-2} yy^\top - q^{-1} I_q \bigr) ,
	\quad y \in \R^q \setminus \{0\} .
\]

%--------------------------------------------------
\subsection{Symmetrized $M$-functionals of scatter}
\label{subsec:Symmetrization}
%--------------------------------------------------

Now we consider a general distribution $P$ on $\R^q$ and want to define its scatter matrix without having to specify a center of $P$. To this end we consider the symmetrized distribution $Q = Q(P)$ as defined in \eqref{eq:symmetrized.P}. Then the symmetrized version of the functional of scatter $\bSigma$ is given by
\[
	\bSigma^{\rm s}(P) \ := \ \bSigma(Q(P)) .
\]
Here we assume that $P$ belongs to the family $\PP$ of all probability distributions on $\R^q$ such that $Q(P) \in \QQ$. In case of an $M$-functional $\bSigma$ with underlying function $\rho$, a sufficient condition for $P \in \PP$ is that
\begin{equation}
\label{eq:P.0.on.hyperplanes}
	P(H) \ = \ 0
	\quad\text{for any hyperplane} \ H \subset \R^q .
\end{equation} 

Analogously, one may define the symmetrized version of Tyler's functional $\bSigma_0$ via $\bSigma_0^{\rm s}(P) := \bSigma_0(Q(P))$, where we assume that $P$ belongs to the family $\PP_0$ of all probability distributions on $\R^p$ such that $Q(P) \in \QQ_0$. Again, condition~\eqref{eq:P.0.on.hyperplanes} is sufficient for that.

As to the benefits of symmetrization, suppose that $P$ is elliptically symmetric with unknown center $\mu_* \in \R^q$ and unknown scatter matrix $\Sigma_* \in \Rqqsympd$. That means, the distribution of $\Sigma_*^{-1/2}(X_1 - \mu_*)$ is spherically symmetric. Then $Q(P)$ is elliptically symmetric with center $0$ and the same scatter matrix $\Sigma_*$. Note that $\Sigma_*$ is defined only up to positive multiples. This is no problem as long as one is mainly interested in the shape matrix $\mathrm{shape}(\Sigma_*)$, where
\[
	\mathrm{shape}(\Sigma) \ := \ \det(\Sigma)_{}^{-1/q} \, \Sigma
\]
for $\Sigma \in \Rqqsympd$, that is, $\mathrm{shape}(\Sigma)$ is a positive multiple of $\Sigma$ with determinant one. For instance, in connection with principal components, regression coefficients and correlation measures, multiplying $\Sigma_*$ with a positive scalar has no impact. Our specific choice of $\mathrm{shape}(\Sigma)$ is justified by \cite{Paindaveine_2008}.

Symmetrization has a second, even more important advantage: Consider an arbitrary distribution $P$, not necessarily symmetric in any sense. Suppose that a random vector $X \sim P$ may be written as $X = [X_{\rm a}^\top, X_{\rm b}^\top]^\top$ with independent subvectors $X_{\rm a} \in \R^{q({\rm a})}$ and $X_{\rm b} \in \R^{q({\rm b})}$. Then $\bSigma^{\rm s}(P)$ is block-diagonal in the sense that
\[
	\bSigma^{\rm s}(P) \ = \ \begin{bmatrix}
		\bSigma_{\rm a}(P) & \bs{0} \\
		\bs{0} & \bs{\Sigma}_{\rm b}(P)
	\end{bmatrix}
\]
with certain matrices $\bs{\Sigma}_{\rm a}(P) \in \R_{\rm sym,+}^{q({\rm a})\times q({\rm a})}$ and $\bs{\Sigma}_{\rm b}(P) \in \R_{\rm sym,+}^{q({\rm b})\times q({\rm b})}$.

At this point, it is not clear whether $\hat{Q}_n$ or $\hat{Q}_{n,d}$ belongs to $\QQ$. As explained in Section~\ref{sec:proofs}, $\hat{Q}_n$ and $\hat{Q}_{n,d}$ converge weakly in probability to $Q$, uniformly in $1 \le d \le (n-1)/2$. Thus, $\Pr(\hat{Q}_n \in \QQ)$ and $\min_{1 \le d \le (n-1)/2} \Pr(\hat{Q}_{n,d} \in \QQ)$ converge to $1$. The same conclusion is true for $(\bSigma_0, \QQ_0)$ in place of $(\bSigma,\QQ)$, if we assume that $P$ has no atoms, that is, if $P(\{x\}) = 0$ for any $x \in \R^q$. Here is also a non-asymptotic result for $M$-estimators of scatter in case of smooth distributions $P$:

\begin{Proposition}
\label{prop:Qhat.in.QQ}
Suppose that $P$ satisfies \eqref{eq:P.0.on.hyperplanes}. With probability one, $\hat{Q}_n(\{0\}) = \hat{Q}_{n,d}(\{0\}) = 0$, and in the case of $n > q$,
\[
	\hat{Q}_n(\mathbb{W}), \hat{Q}_{n,d}(\mathbb{W}) \ < \ \frac{\dim(\mathbb{W})}{q}
\]
for arbitrary linear subspaces $\mathbb{W}$ of $\R^q$ with $1 \le \dim(\mathbb{W}) < q$ and $1 \le d \le (n-1)/2$.
\end{Proposition}

This theorem implies that for the $M$-functional $\bSigma(\cdot)$, the symmetrized $M$-estimators $\bSigma(\hat{Q}_n)$ and $\bSigma(\hat{Q}_{n,d})$ are well-defined almost surely for $1 \le d \le (n-1)/2$, provided that $n > q$ and $P$ satisfies \eqref{eq:P.0.on.hyperplanes}. The same conclusion is true for Tyler's $M$-functional $\bSigma_0(\cdot)$ in place of $\bSigma(\cdot)$.

%================================================
\section{Asymptotic Expansions and Distributions}
\label{sec:asymptotics}
%================================================

In what follows, $\bSigma(\cdot)$ denotes either a linear equivariant and differentiable scatter functional or Tyler's functional $\bSigma_0(\cdot)$. In addition to $\hat{Q}_n$ and $\hat{Q}_{n,d}$, we consider the usual empirical distribution of the observations $X_i$,
\[
	\hat{P}_n \ := \ n^{-1} \sum_{i=1}^n \delta_{X_i}^{} .
\]

\begin{Theorem}
\label{thm:asymptotics}
Suppose that $\bSigma(Q)$ is well-defined for $Q = Q(P)$. With $J = J_Q$, define
\[
	H_1(x) \ := \ \Ex J(x - X_1)
	\quad\text{and}\quad
	H_2(x,y) \ := \ J(x - y) - H_1(x) - H_1(y)
\]
for $x, y \in \R^q$, where $J(0) := 0$ in connection with Tyler's functional. Let $\bs{G}_1$ and $\bs{G}_2$ be two stochastically independent Gaussian random matrices in $\Rqqsym$ such that $\Ex \bs{G}_1 = \Ex \bs{G}_2 = 0$, and
\begin{align*}
	\Ex \bigl( \langle A, \bs{G}_1\rangle^2 \bigr) \
	&= \ \Ex \bigl( \langle A, H_1(X_1) \rangle^2 \bigr) , \\
	\Ex \bigl( \langle A, \bs{G}_2\rangle^2 \bigr) \
	&= \ \Ex \bigl( \langle A, H_2(X_1,X_2) \rangle^2 \bigr)
\end{align*}
for all matrices $A \in \Rqqsym$. If $(n-1)/2 \ge d(n) \to \infty$, then
\[
	\left.\begin{array}{l}
		\bSigma(\hat{Q}_n) \\[1ex]
		\bSigma(\hat{Q}_{n,d(n)})
	\end{array} \right\} \
	= \ \bSigma(Q) + 2 \int_{\R^q} H_1 \, d\hat{P}_n + o_p(n^{-1/2}) .
\]
For fixed integers $d \ge 1$,
\[
	\bSigma(\hat{Q}_{n,d}) \ = \ \bSigma(\hat{Q}_n) + \bs{M}_{n,d} + o_p(n^{-1/2}) ,
\]
where
\[
	\bs{M}_{n,d} \ := \ (nd)^{-1} \sum_{i=1}^n \sum_{j=i+1}^{i+d}
		H_2(X_i,X_j) .
\]
Moreover,
\[
	\Bigl( \sqrt{n} \int_{\R^q} H_1 \, d\hat{P}_n, \, \sqrt{nd} \, \bs{M}_{n,d} \Bigr)
	\ \to_{\LL}^{} \ (\bs{G}_1, \bs{G}_2) .
\]
In particular, as $d(n) \to \infty$,
\[
	\left.\begin{array}{l}
		\sqrt{n} \bigl( \bSigma(\hat{Q}_n) - \bSigma(Q) \bigr) \\[1ex]
		\sqrt{n} \bigl( \bSigma(\hat{Q}_{n,d(n)}) - \bSigma(Q) \bigr)
	\end{array} \right\} \
	\to_{\LL}^{} \ 2 \bs{G}_1 ,
\]
whereas for fixed integers $d \ge 1$,
\[
	\sqrt{n} \bigl( \bSigma(\hat{Q}_{n,d}) - \bSigma(Q) \bigr) \
	\to_{\LL}^{} \ 2 \bs{G}_1 + d^{-1/2} \bs{G}_2 .
\]
\end{Theorem}

It remains to explain the asymptotic properties of the alternative estimator $d^{-1} \sum_{\ell=1}^d \bSigma(\hat{Q}_{n,1}^{(\ell)})$, where $\hat{Q}_{n,1}^{(\ell)}$ is defined as $\hat{Q}_{n,1}$ with $(X_{\Pi^{(\ell)}(i)})_{i=1}^n$ in place of $(X_i)_{i=1}^n$. Here $\Pi^{(1)}, \ldots, \Pi^{(d)}$ are independent random permutations of $\{1,2,\ldots,n\}$, and independent from the data $(X_i)_{i=1}^n$.

\begin{Theorem}
\label{thm:asymptotics.2}
For fixed $d \ge 1$ and $1 \le \ell \le d$,
\[
	\bSigma(\hat{Q}_{n,1}^{(\ell)}) \ = \ \bSigma(Q) + 2 \int_{\R^q} H_1 \, d\hat{P}_n
		+ M_{n,1}^{(\ell)} + o_p(n^{-1/2}) ,
\]
where
\[
	M_{n,1}^{(\ell)} \ := \ n^{-1} \sum_{i=1}^n H_2(X_{\Pi^{(\ell)}(i)},X_{\Pi^{(\ell)}(i+1)})
\]
with $\Pi^{(\ell)}(n+1) := \Pi^{(\ell)}(1)$. Moreover,
\[
	\sqrt{n} \Bigl( \int_{\R^q} H_1 \, d\hat{P}_n, \, M_{n,1}^{(1)}, \ldots, M_{n,1}^{(d)} \Bigr) \
	\to_{\LL}^{} \ \bigl( \bs{G}_1, \, \bs{G}_2^{(1)}, \ldots, \bs{G}_2^{(d)} \bigr)
\]
with independent random matrices $\bs{G}_1$ and $\bs{G}_2^{(1)}, \ldots, \bs{G}_2^{(d)}$, where $\bs{G}_1$ and $\bs{G}_2^{(\ell)}$ have the same distribution as $\bs{G}_1$ and $\bs{G}_2$, respectively, in Theorem~\ref{thm:asymptotics}. In particular,
\[
	\sqrt{n} \Bigl( d^{-1} \sum_{\ell=1}^d \bSigma(\hat{Q}_{n,1}^{(\ell)}) - \bSigma(Q) \Bigr)
	\ \to_{\LL}^{} \ 2 \bs{G}_1 + d^{-1/2} \bs{G}_2 .
\]
\end{Theorem}

This theorem shows that averaging $\bSigma(\hat{Q}_{n,1}^{(\ell)})$ over $\ell = 1, \ldots, d$ is asymptotically equivalent to computing $\bSigma(\hat{Q}_{n,d})$. One could guess that averaging over $d(n)$ random permutations with $d(n) \to \infty$ leads to an estimator with the same asymptotic distributions as $\bSigma(\hat{Q}_n)$. But this is not obvious, because the average of $d(n)$ random variables which are uniformly of order $o_p(1)$ need not be of order $o_p(1)$ too.

%===============================
\section{Numerical Illustration}
\label{sec:examples}
%===============================

The computations are based on Partial Newton algorithms proposed by \cite{Duembgen_etal_2016}. They are implemented in the R package \emph{fastM} by \cite{fastMXXX} which is publicly available on CRAN.

%%---------------------------------------
%\subsection{Estimating the shape matrix}
%%---------------------------------------

As explained in Section~\ref{subsec:Symmetrization}, in numerous applications one is mainly interested in the scatter matrix up to positive scalars. Thus we illustrate the previous results with the shape matrix $\bs{H} := \mathrm{shape}(\bSigma(Q))$ and its estimators
\begin{align*}
	\hat{\bs{H}}_n \
	&:= \ \mathrm{shape}(\bSigma(\hat{Q}_n)) , \\
	\hat{\bs{H}}_{n,d} \
	&:= \ \mathrm{shape}(\bSigma(\hat{Q}_{n,d})) , \\
	\hat{\bs{H}}_{n,d}^{\rm rand} \
	&:= \ \mathrm{shape} \Bigl( d^{-1} \sum_{\ell=1}^d \bSigma(\hat{Q}_{n,1}^{(\ell)}) \Bigr) .
\end{align*}
On the one hand, we look at the approximation errors, that is, the distances between $\hat{\bs{H}}_{n,d}^{}, \hat{\bs{H}}_{n,d}^{\rm rand}$ and the full estimator $\hat{\bs{H}}_n^{}$. The distance between two matrices $\Sigma_1, \Sigma_2 \in \Rqqsympd$ is measured by the so-called geodesic distance
\[
	D(\Sigma_1,\Sigma_2) \ := \ \Bigl( \sum_{j=1}^q \log[\lambda_j(\Sigma_1^{-1} \Sigma_2^{})]^2 \Bigr) ,
\]
where $\lambda_1(\cdot) \ge \cdots \ge \lambda_q(\cdot)$ are the ordered real eigenvalues of a matrix $A \in \Rqq$. On the other hand, we look at the estimation errors, that is, the distances between the estimators $\hat{\bs{H}}_n^{}$, $\hat{\bs{H}}_{n,d}^{}$, $\hat{\bs{H}}_{n,d}^{\rm rand}$ and the true shape matrix $\bs{H}$.

We simulated $2000$ times a data set of size $n = 100$ in dimension $q = 10$, where each observation $X_i$ had independent components with standard exponential distribution. The scatter functional was the $M$-functional with $\rho(s) = \rho_1(s) = (q+1) \log(s + 1)$ for $s \ge 0$. In this particular example, $\bSigma(P)$ is \textsl{not} a multiple of $I_q$, but the symmetrized distributions $Q = Q(P)$ yields $\bs{H} = I_q$. Figures~\ref{fig:Approximation_errors_A_100} and \ref{fig:Approximation_errors_B_100} show box-and-whiskers plots of the resulting relative approximation errors
\[
	D(\hat{\bs{H}}_{n,d}^{},\hat{\bs{H}}_n^{})/D(\hat{\bs{H}}_n^{},\bs{H})
	\quad\text{and}\quad
	D(\hat{\bs{H}}_{n,d}^{\rm rand},\hat{\bs{H}}_n^{})/D(\hat{\bs{H}}_n^{},\bs{H}) ,
\]
respectively, for $1 \le d \le 49$. Figure~\ref{fig:Approximation_errors_AB_100} shows these ratios in one plot for $1 \le d \le 15$.

\begin{figure}
\centering
\includegraphics[width=0.9\textwidth]{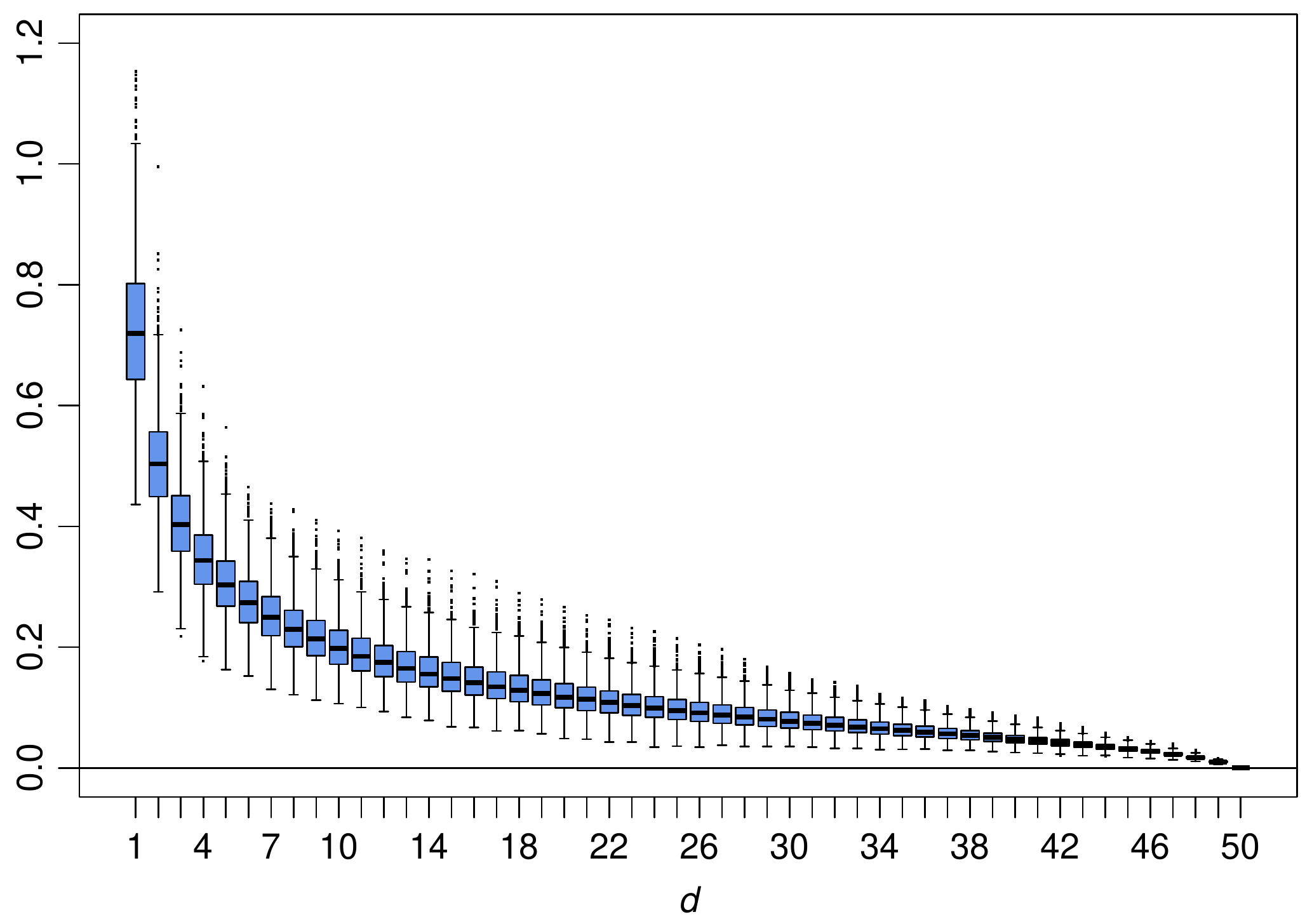}
\caption{$(q,n) = (10,100)$: Relative approximation errors $D(\hat{\bs{H}}_{n,d}^{},\hat{\bs{H}}_n^{})/D(\hat{\bs{H}}_n^{},\bs{H})$.}
\label{fig:Approximation_errors_A_100}
\end{figure}

\begin{figure}
\centering
\includegraphics[width=0.9\textwidth]{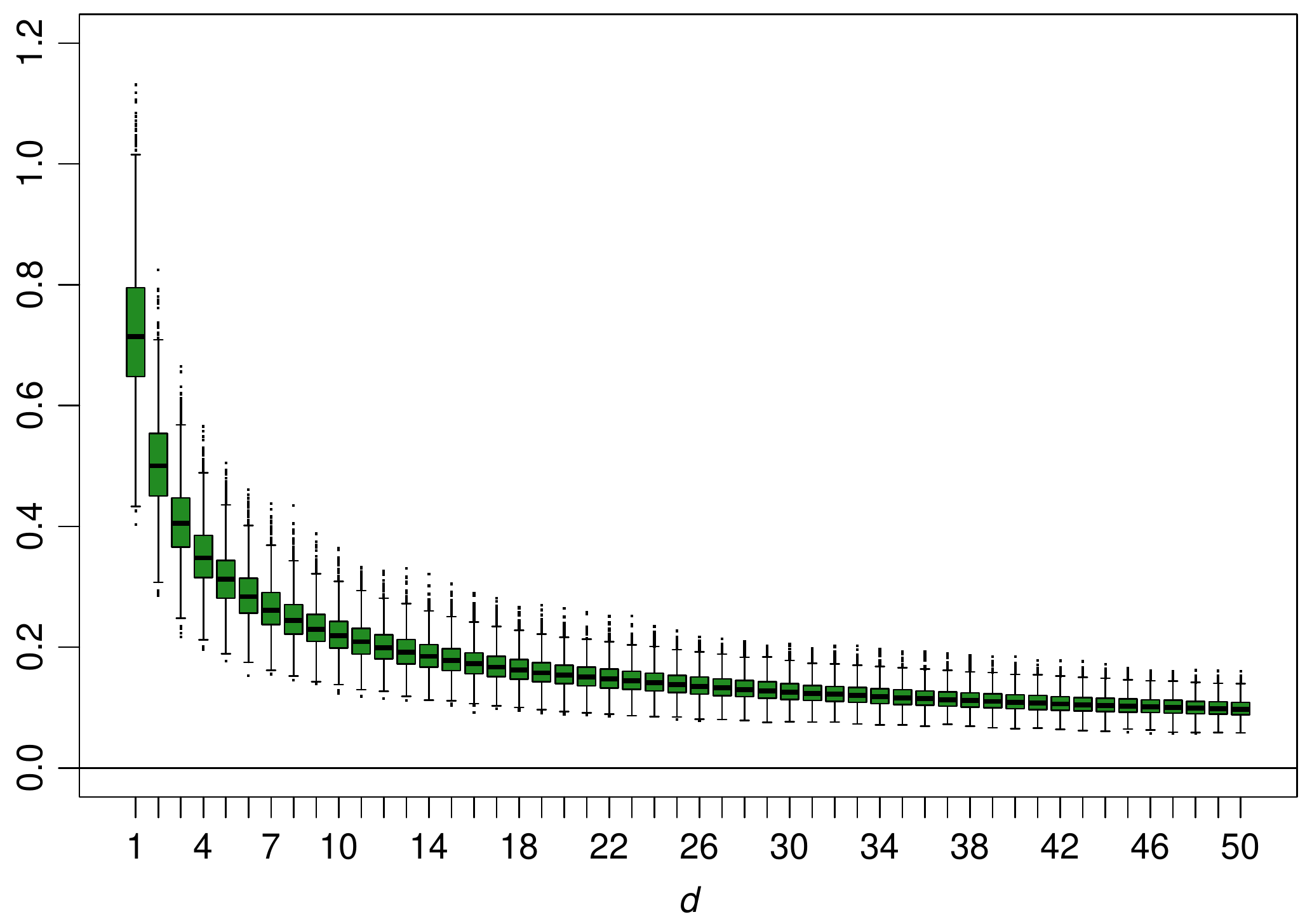}
\caption{$(q,n) = (10,100)$: Relative approximation errors $D(\hat{\bs{H}}_{n,d}^{\rm rand},\hat{\bs{H}}_n^{})/D(\hat{\bs{H}}_n^{},\bs{H})$.}
\label{fig:Approximation_errors_B_100}
\end{figure}

\begin{figure}
\centering
\includegraphics[width=0.9\textwidth]{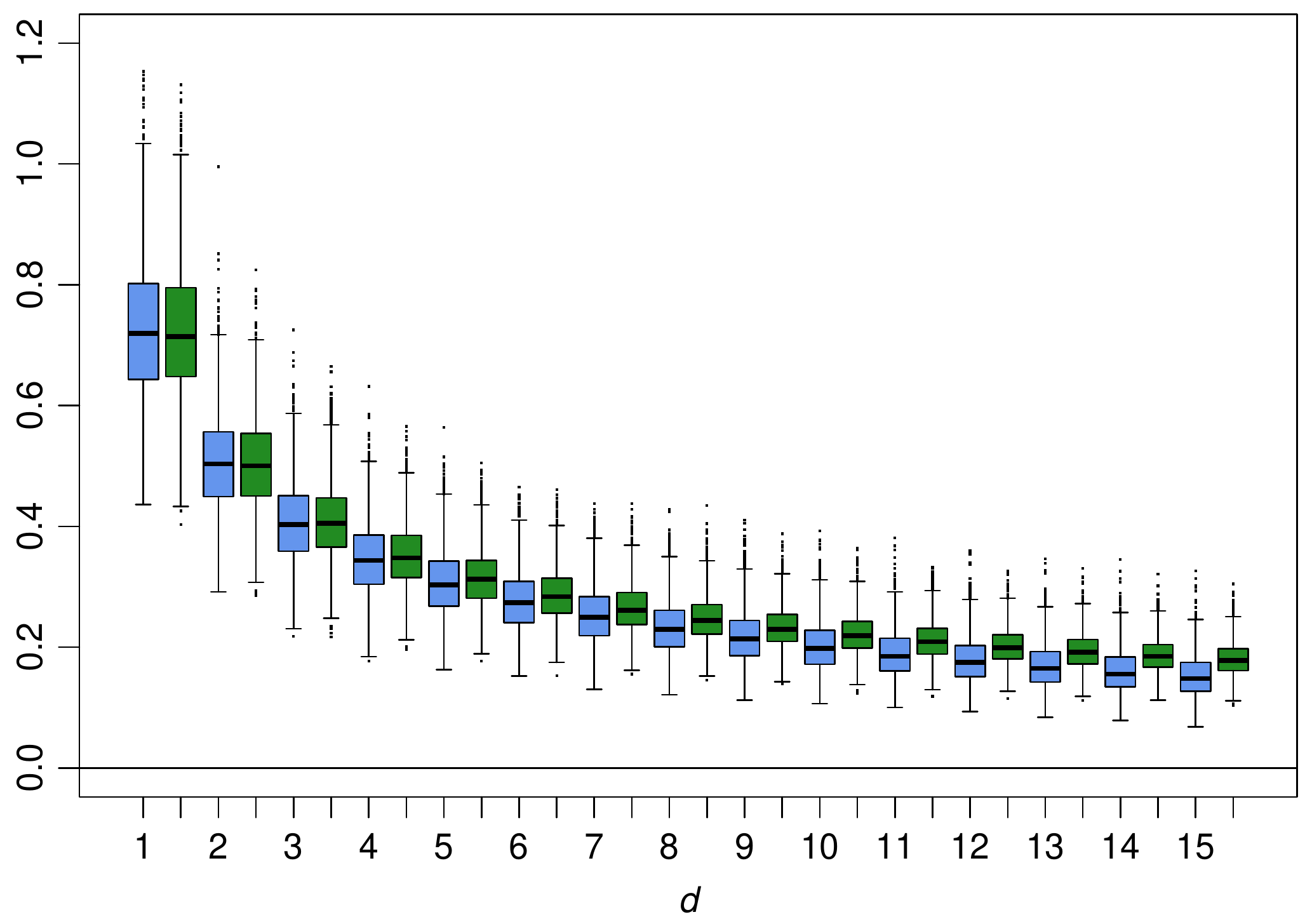}
\caption{$(q,n) = (10,100)$: Relative approximation errors $D(\hat{\bs{H}}_{n,d}^{},\hat{\bs{H}}_n^{})/D(\hat{\bs{H}}_n^{},\bs{H})$ (blue) and $D(\hat{\bs{H}}_{n,d}^{\rm rand},\hat{\bs{H}}_n^{})/D(\hat{\bs{H}}_n^{},\bs{H})$ (green).}
\label{fig:Approximation_errors_AB_100}
\end{figure}

Figures~\ref{fig:Estimation_errors_A_100}, \ref{fig:Estimation_errors_B_100} and \ref{fig:Estimation_errors_AB_100} are analogous, this time with the relative estimation errors
\[
	D(\hat{\bs{H}}_{n,d}^{},\bs{H})/D(\hat{\bs{H}}_n^{},\bs{H})
	\quad\text{and}\quad
	D(\hat{\bs{H}}_{n,d}^{\rm rand},\bs{H})/D(\hat{\bs{H}}_n^{},\bs{H}) .
\]
The median of the estimation error $D(\hat{\bs{H}}_n^{},\bs{H})$ in the simulations was equal to $1.1643$. Interestingly, the relative estimation errors approach $1$ more quickly than the relative approximation arrors approach $0$. With respect to relative estimation error, a value $d = 10$, say, seems to be sufficient, although the approximation errors for this value are still substantial.

\begin{figure}
\centering
\includegraphics[width=0.95\textwidth]{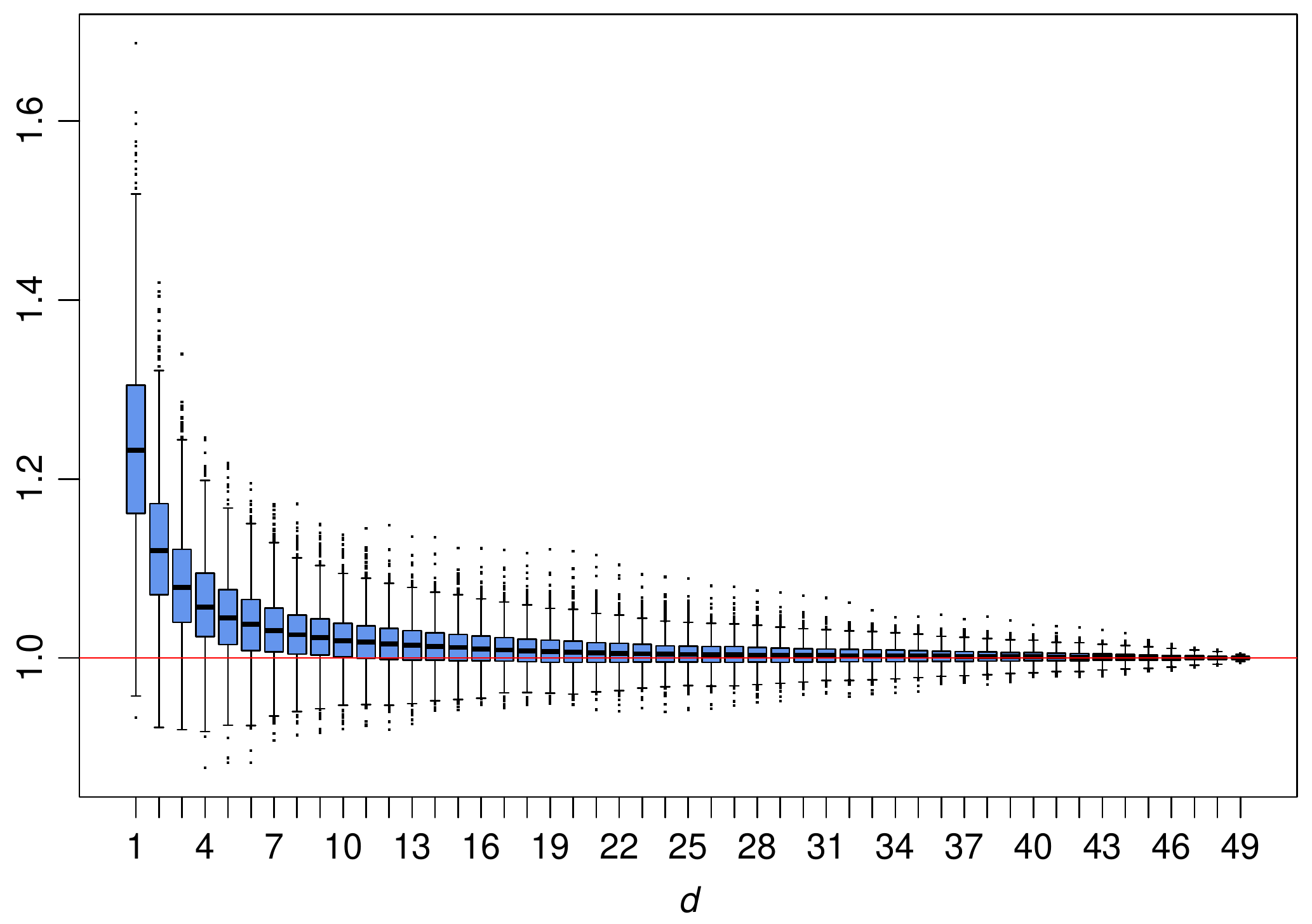}
\caption{$(q,n) = (10,100)$: Relative estimation errors $D(\hat{\bs{H}}_{n,d}^{},\bs{H})/D(\hat{\bs{H}}_n^{},\bs{H})$.}
\label{fig:Estimation_errors_A_100}
\end{figure}

\begin{figure}
\centering
\includegraphics[width=0.95\textwidth]{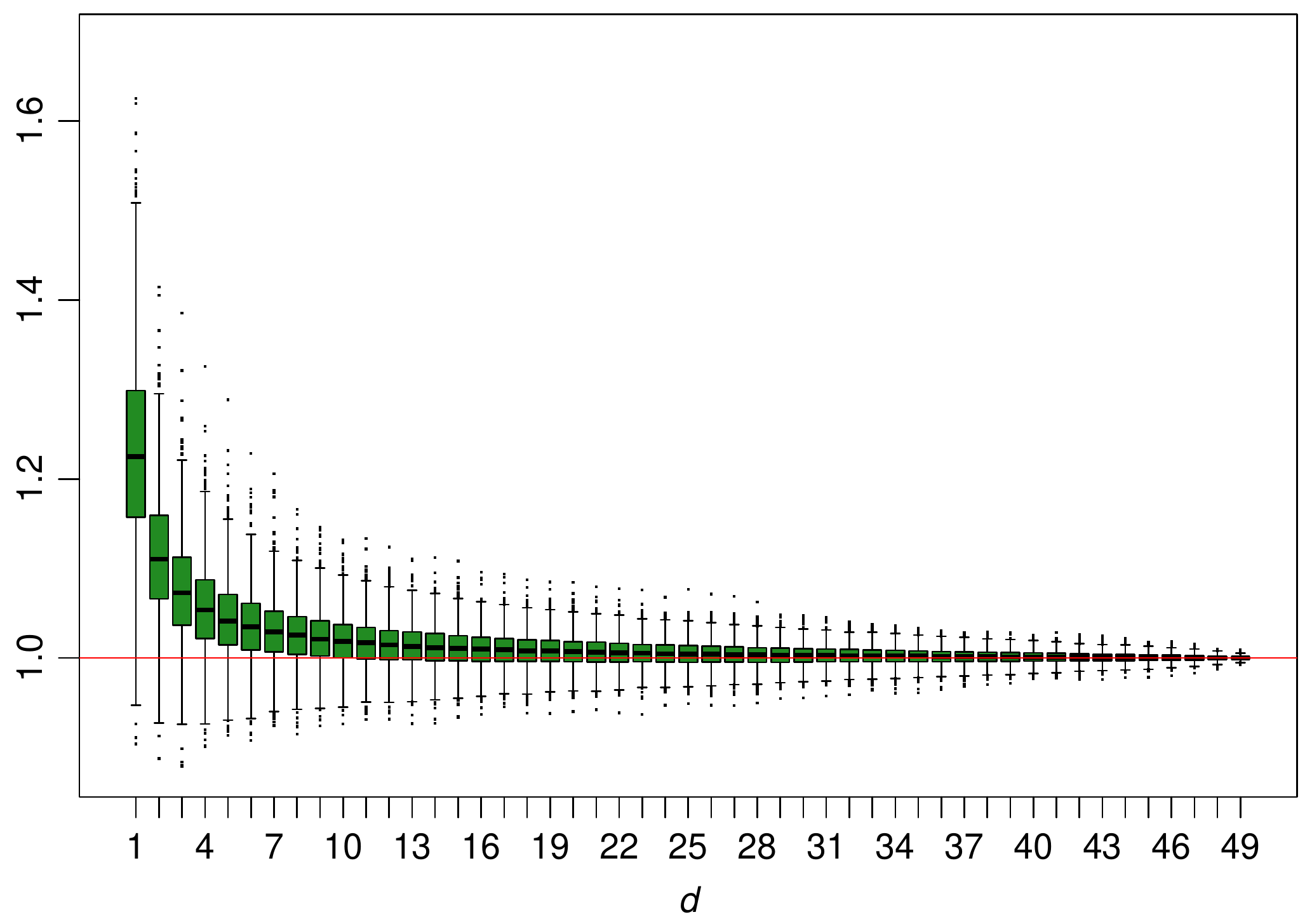}
\caption{$(q,n) = (10,100)$: Relative estimation errors $D(\hat{\bs{H}}_{n,d}^{\rm rand},\bs{H})/D(\hat{\bs{H}}_n^{},\bs{H})$.}
\label{fig:Estimation_errors_B_100}
\end{figure}

\begin{figure}
\centering
\includegraphics[width=0.9\textwidth]{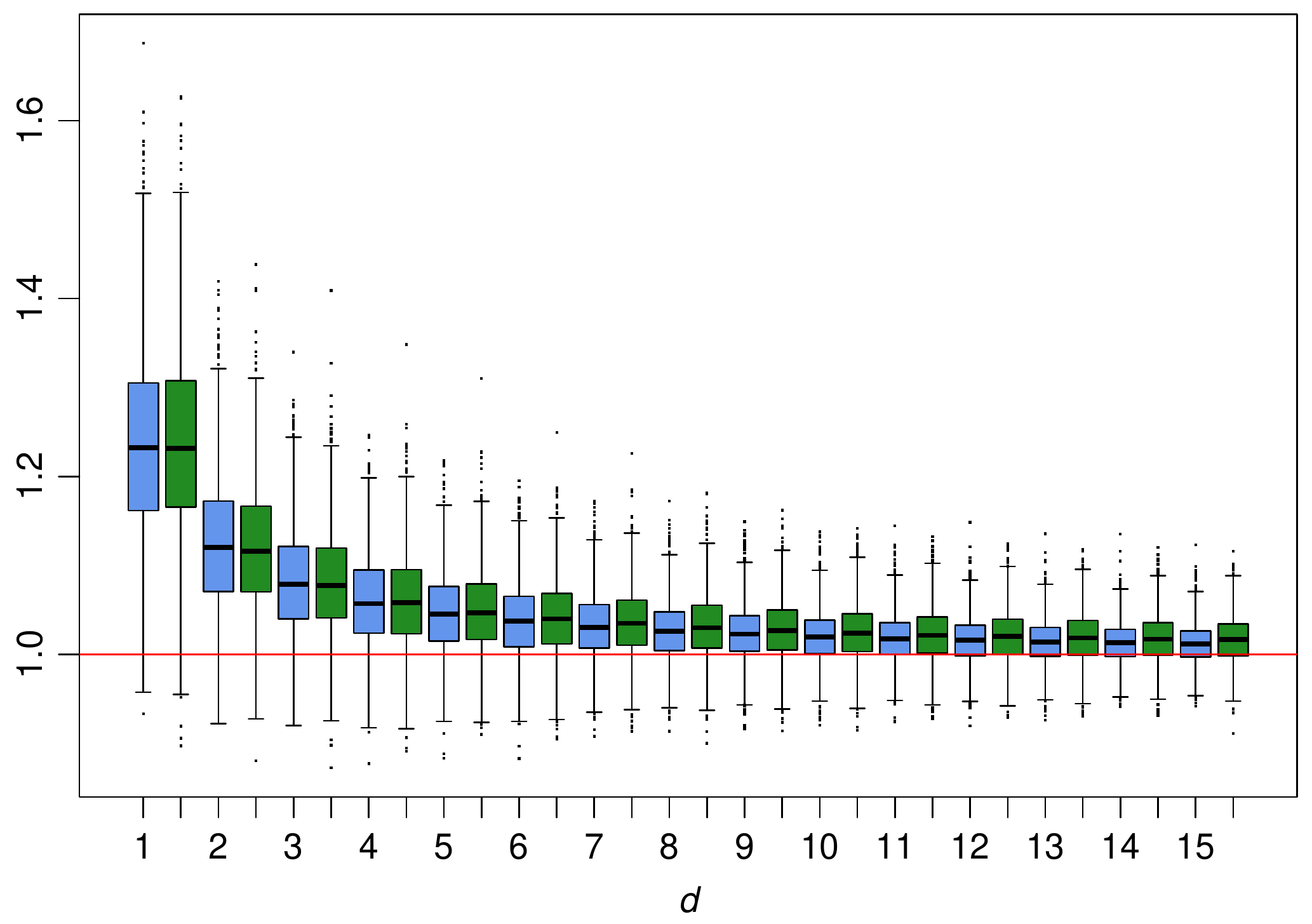}
\caption{$(q,n) = (10,100)$: Relative estimation errors $D(\hat{\bs{H}}_{n,d}^{},\bs{H})$ (blue) and $D(\hat{\bs{H}}_{n,d}^{\rm rand},\bs{H})$ (green).}
\label{fig:Estimation_errors_AB_100}
\end{figure}

We did the same simulations and calculations for sample size $n = 400$ instead of $n = 100$. Figures~\ref{fig:Approximation_errors_AB_400} and \ref{fig:Estimation_errors_AB_400} show the resulting relative approximation errors and relative estimation errors. This time, the median of $D(\hat{\bs{H}}_n^{}, \bs{H})$ was only $0.5662$. But note that the relative errors are similarly distributed for both sample sizes. The main difference seems to be that with increasing sample size the differences between $\hat{\bs{H}}_{n,d}$ and $\hat{\bs{H}}_{n,d}^{\rm rand}$ become smaller. 

The simulation results are coherent with the asymptotic theory and confirm our claim that moderately large values of $d$ yield already estimators with similar precision as the full symmetrized $M$-estimators. Therefore for larger sample sizes, computational costs are no longer a hindrance to apply symmetrized scatter matrices in practice.

\begin{figure}
\centering
\includegraphics[width=0.9\textwidth]{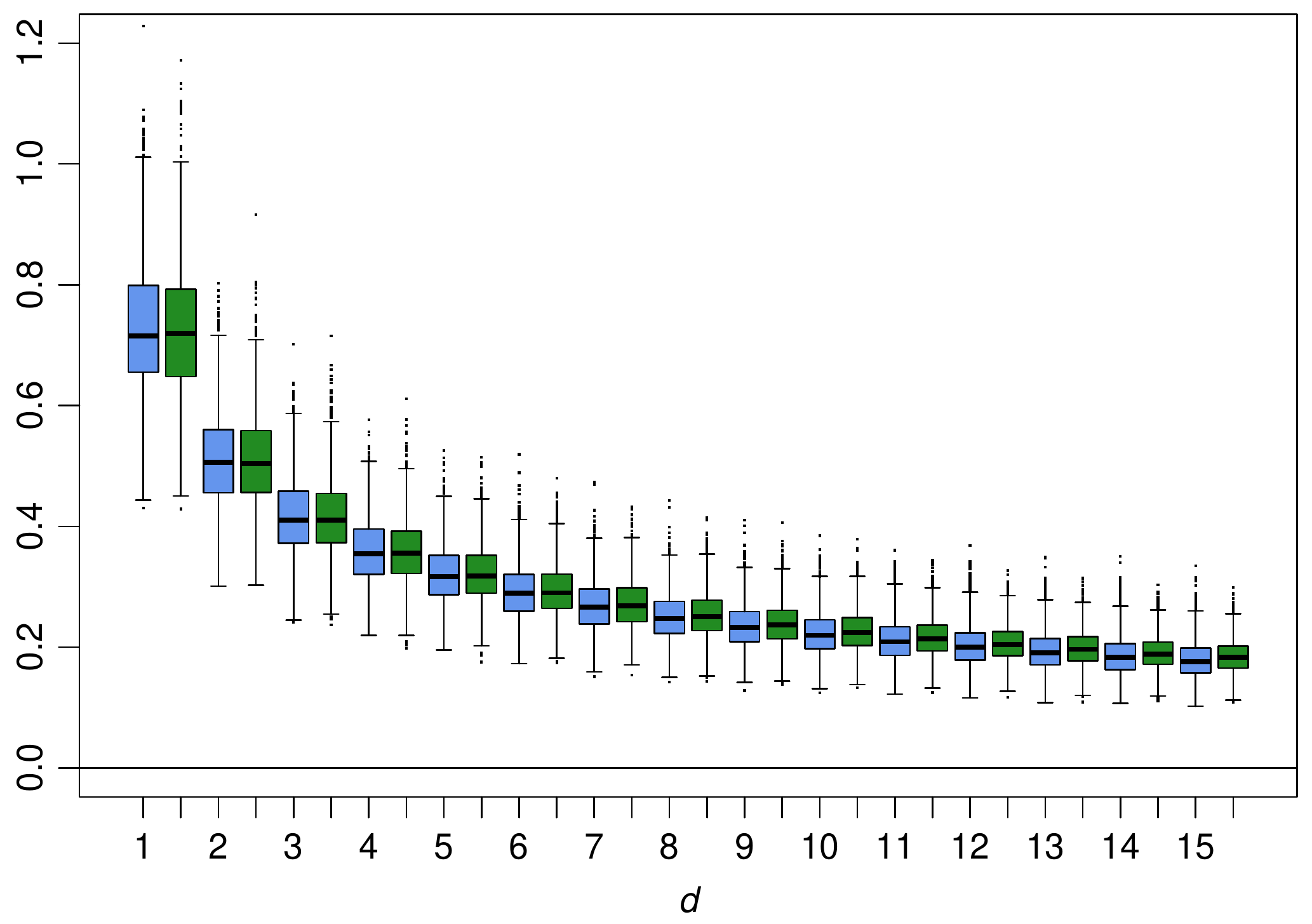}
\caption{$(q,n) = (10,400)$: Relative approximation errors $D(\hat{\bs{H}}_{n,d}^{},\hat{\bs{H}}_n^{})/D(\hat{\bs{H}}_n^{},\bs{H})$ (blue) and $D(\hat{\bs{H}}_{n,d}^{\rm rand},\hat{\bs{H}}_n^{})/D(\hat{\bs{H}}_n^{},\bs{H})$ (green).}
\label{fig:Approximation_errors_AB_400}
\end{figure}

\begin{figure}
\centering
\includegraphics[width=0.9\textwidth]{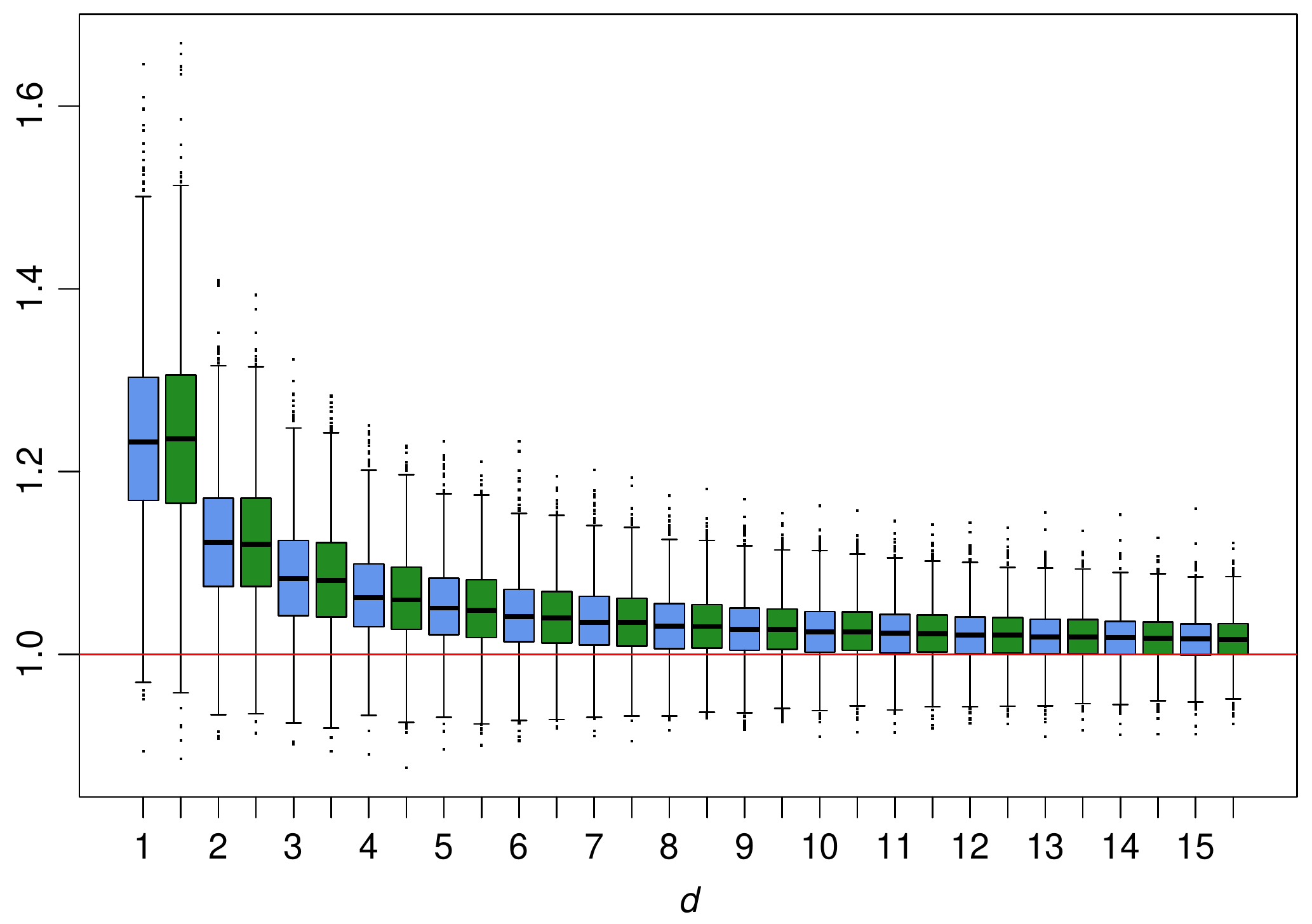}
\caption{$(q,n) = (10,400)$: Relative estimation errors $D(\hat{\bs{H}}_{n,d}^{},\bs{H})/D(\hat{\bs{H}}_n^{},\bs{H})$ (blue) and $D(\hat{\bs{H}}_{n,d}^{\rm rand},\bs{H})/D(\hat{\bs{H}}_n^{},\bs{H})$ (green).}
\label{fig:Estimation_errors_AB_400}
\end{figure}

%%------------------------------------------
%\subsection{Independent component analysis}
%%------------------------------------------
%
%???

%=================
\section{Proofs}
\label{sec:proofs}
%=================

\begin{proof}[\bf Proof of Proposition~\ref{prop:Qhat.in.QQ}]
Some of our arguments are similar to parts of Section~8.2 of \cite{Duembgen_etal_2015}, but for the reader's convenience, we present a complete and self-contained proof here.

\textbf{Step~0.}
For arbitrary different indices $i,j \in \{1,\ldots,n\}$, the vector $X_j - X_i \ne 0$ almost surely, because $\Pr(X_j - X_i = 0) = \Ex \Pr(X_j \in \{X_i\} \,|\, X_i) = 0$. Hence, $\hat{Q}_n(\{0\}) = \hat{Q}_{n,d}(\{0\}) = 0$ for $1 \le d \le (n-1)/2$.

\textbf{Step~1.}
Let $\SS$ be the set of all index sets $\{i,j\}$, $1 \le i < j \le n$. Let $\EE_o \subset \SS$, and set $V_o := \bigcup_{E \in \EE_o} E$. Suppose that the graph $(V_o,\EE_o)$ is connected. That means, for arbitrary $\{i,j\} \in \EE$, there exist $T \in \mathbb{N}$ and indices $i_0, i_1, \ldots, i_T$ in $V_o$ such that $i_0 = i$, $i_T = j$, and $\{i_{t-1},i_t\} \in \EE_o$ for $1 \le t \le T$. Then for any index $i_o \in V_o$, the following three linear spaces are identical:
\begin{align*}
	\mathbb{W}_1 \
	&:= \ \mathrm{span}(X_i - X_{i_o} : i \in V_o) , \\
	\mathbb{W}_2 \
	&:= \ \mathrm{span}(X_j - X_i : \{i, j\} \in \EE_o) , \\
	\mathbb{W}_3 \
	&:= \ \mathrm{span}(X_j - X_i : i,j \in V_o) .
\end{align*}
The inclusions $\mathbb{W}_1, \mathbb{W}_2 \subset \mathbb{W}_3$ are obvious. On the other hand, for $i,j \in V_o$, the vector $X_j - X_i = (X_j - X_{i_o}) - (X_i - X_{i_o}) \in \mathbb{W}_1$, whence $\mathbb{W}_3 \subset \mathbb{W}_1$. Finally, by connectednes of $(V_o,\EE_o)$, for arbitrary different indices $i,j \in V_o$, there exist $T \in \mathbb{N}$ and indices $i_0, i_1, \ldots, i_T \in V_o$ such that $i_0 = i$, $i_T = j$, and $\{i_{t-1},i_t\} \in \EE_o$ for $1 \le t \le T$. Hence, $X_j - X_i = \sum_{t=1}^T (X_{i_t} - X_{i_{t-1}}) \in \mathbb{W}_2$, and this shows that $\mathbb{W}_3 \subset \mathbb{W}_2$.

\textbf{Step~2.}
Let $\EE$ be an arbitrary subset of $\SS$, and let $V := \bigcup_{E \in \EE} E$. Let $V_1, \ldots, V_M$ be the $M \ge 1$ maximal connected components of the graph $(V,\EE)$. That means, $\EE = \bigcup_{m=1}^M \EE_m$ with sets $\EE_m \subset \SS$ such that the sets $V_m := \bigcup_{E \in \EE_m} E$ are disjoint, and each subgraph $(V_m,\EE_m)$ is connected. Then, the linear space
\[
	\mathbb{W} \ := \ \mathrm{span}(X_i - X_j : \{i,j\} \in \EE)
\]
has almost surely dimension
\[
	\dim(\mathbb{W}) \ = \ \min(S,q)
	\quad\text{with}\quad
	S \ := \ \sum_{m=1}^M (\#V_m - 1) .
\]
To verify this, fix an arbitrary point $i_m \in V_m$ for $1 \le m \le M$. Then Step~1 shows that
\[
	\mathbb{W} \ = \ \sum_{m=1}^M \mathrm{span}(X_j - X_{i_m} : j \in V_m \setminus \{i_m\}) ,
\]
and it suffices to show that in case of $S \le q$, the vectors $X_j - X_{i_m}$, $j \in V_m \setminus \{i_m\}$, $1 \le m \le M$, are almost surely linearly independent. But this can be shown by induction: Let $\bigl\{ \{i_m,j\} : 1 \le m \le M, j \in V_m \setminus \{i_m\} \bigr\} = \{(k_1,\ell_1), \ldots, (k_S,\ell_S)\}$ with $k_1,\ldots,k_S \in \{i_1,\ldots,i_M\}$ and $\ell_s \in \bigcup_{m=1}^M V_m \setminus \{i_m\}$. Then, by Step~0, $X_{\ell_1} - X_{k_1} \ne 0$ almost surely, and for $1 \le s < S$ and $\mathbb{W}_s := \mathrm{span}(X_{\ell_r} - X_{k_r} : 1 \le r \le s)$,
\begin{align*}
	\Pr( & X_{\ell_{s+1}} - X_{k_{s+1}} \not\in \mathbb{W}_s) \\
	&= \ \Ex \Pr \bigl( X_{\ell_{s+1}} \not\in X_{k_{s+1}} + \mathbb{W}_s \,\big|\,
		X_i : i \in \{i_1,\ldots,i_M\} \cup \{\ell_1,\ldots,\ell_s\} \bigr) \ = \ 0 .
\end{align*}

\textbf{Step~3.}
With $(V,\EE)$ and its subgraphs $(V_m,\EE_m)$, $1 \le m \le M$, as in Step~2,
\[
	\# \EE \ \le \ \sum_{m=1}^M \binom{\# V_m}{2} \ \le \ \binom{S+1}{2} .
\]
The first inequality is a consequence of $\#\EE_m \le \binom{\# V_m}{2}$ for $1 \le m \le M$. The second inequality follows from the fact that the mapping
\[
	\EE \ni \{i,j\} \ \mapsto \ \begin{cases}
		\{i,j\} & \text{for} \ i,j \in V_m \setminus \{i_m\}, 1 \le m \le M , \\
		\{0,j\} & \text{for} \ i = i_m, j \in V_m \setminus \{i_m\}, 1 \le m \le M ,
	\end{cases}
\]
is injective, and the images are subsets of $\{0\} \cup \bigcup_{m=1}^M V_m \setminus \{i_m\}$ with two elements.

For a fixed integer $d \ge 1$ with $d \le (n-1)/2$, let $\SS_d$ be the subset of all $\{i,j\} \in \SS$ such that $0 < j-i \le d$ or $j - i \ge n-d$. That means, for any $i \in \{1,\ldots,n\}$ there are exactly $2d$ indices $j \in \{1,\ldots,n\}$ such that $\{i,j\} \in \SS_d$. Then
\[
	\# (\EE \cap \SS_d) \ \le \ S d
	\quad\text{unless} \ M = 1 \ \text{and} \ V = \{1,2,\ldots,n\} .
\]
To see this, note that in case of $M > 1$ or $V \ne \{1,2,\ldots,n\}$, all sets $V_m$ are different from $\{1,\ldots,n\}$. For a given $m \in \{1,\ldots,M\}$, let $k \in \{1,\ldots,n\} \setminus V_m$. To get an upper bound for $\# (\EE_m \cap \SS_d)$, we may assume without loss of generality that $k = n$. Otherwise, we could transform $\{1,2,\ldots,n\}$ with the permutation $i \mapsto T(i) := 1_{[i \le k]} (i + n - k) + 1_{i > k} (i - k)$, because $\{i,j\} \in \SS_d$ if and only if $\{T(i),T(j)\} \in \SS_d$. Now, if $i_0 < i_1 < \cdots < i_{q_m} < n$ are the elements of $V_m$, then
\begin{align*}
	\#(\EE_m \cap \SS_d) \
	= \ &\# \bigl\{ \{i_a,i_b\} : 0 \le a < b \le q_m, i_b - i_a \le d \ \text{or} \ i_b - i_a \ge n-d \bigr\} \\
	\le \ &\# \bigl\{ \{a,b\} : 0 \le a < b \le q_m, b - a \le d \bigl\} \\
		&+ \ \# \bigl\{ \{i,j\} : 1 \le i < j < n, j-i \ge n-d \bigr\} \\
	= \ &\# \bigl\{ \{a,a+c\} : 1 \le c \le d, 0 \le a \le q_m - c \bigl\} \\
		&+ \ \# \bigl\{ \{i,j\} : 1 \le i < d, n-d+i \le j < n \bigr\} \\
	\le \ &\sum_{c=1}^d (q_m + 1 - c) + \sum_{i=1}^{d-1} (d-i) \\
	= \ &q_m d - \sum_{c' = 0}^{d-1} c' + \sum_{i' = 1}^{d-1} i'
		\ = \ q_m d \ = \ (\# V_m - 1) d .
\end{align*}

\textbf{Step~4.}
Since there are only finitely many nonempty subsets $\EE$ of $\SS$, we may conclude from Step~2 that for \textsl{any} nonempty set $\EE \subset \SS$, the dimension of $\mathrm{span}(X_j - X_i : \{i,j\} \in \EE)$ is given by $S = S(\EE)$ as defined in Step~2. Now we consider an arbitrary linear subspace $\mathbb{W}$ of $\R^q$ with dimension $q' < q$ such that $\EE = \EE(\mathbb{W}) := \bigl\{ \{i,j\} \in \SS : X_j - X_i \in \mathbb{W} \bigr\}$ is nonempty. Then Step~3 implies that
\[
	\hat{Q}_n(\mathbb{W}) \
	\le \ \binom{n}{2}^{-1} \binom{q'+1}{2}
	\quad\text{and}\quad
	\hat{Q}_{n,d}(\mathbb{W}) \
	\le \ \frac{q'}{n} .
\]
But
\[
	\binom{n}{2}^{-1} \binom{q'+1}{2} \
	= \ \frac{q'}{q} \frac{q(q'+1)}{n(n-1)} \
	\le \ \frac{q'}{q} \frac{q^2}{n(n-1)}
	\quad\text{and}\quad
	\frac{q'}{n} \ = \ \frac{q'}{q} \frac{q}{n} .
\]
Both factors $q^2/(n(n-1))$ and $q/n$ are strictly smaller than $1$ if and only if $n > q$. This proves our claim about $\hat{Q}_n$ and $\hat{Q}_{n,d}$.
\end{proof}

\paragraph{Some facts about complete and balanced incomplete $U$-statistics.}
Let us first recollect some well-known facts about $U$-statistics of order two \citep{Serfling_1980,Lee_1990}, with obvious adaptations to vector-valued kernels and the particular distributions $\hat{Q}_n$ and $\hat{Q}_{n,d}$. For some integer $r \ge 1$, let $f : \R^q \to \R^r$ be measurable such that $\Ex(\|f(X_1 - X_2)\|^2) < \infty$. With the symmetrized function $f^{\rm s}(x) := 2^{-1} \bigl( f(x) + f(-x) \bigr)$, define $f_0 := \Ex f(X_1 - X_2) = \Ex f^{\rm s}(X_1 - X_2)$ and
\[
	f_1(x) \ := \ \Ex f^{\rm s}(x - X_1) - f_0 ,
	\quad
	f_2(x,y) \ := \ f^{\rm s}(x-y) - f_0 - f_1(x) - f_1(y)
\]
for $x,y \in \R^q$. Then the covariance matrices $\Gamma := \Var(f(X_1 - X_2))$, $\Gamma^{\rm s} := \Var(f^{\rm s}(X_1 - X_2))$, $\Gamma_1 := \Var(f_1(X_1))$ and $\Gamma_2 := \Var(f_2(X_1,X_2))$ satisfy the (in)equalities
\[
	\Gamma \ \ge \ \Gamma^{\rm s} \ = \ 2 \Gamma_1 + \Gamma_2 .
\]
Here and subsequently, inequalities between symmetric matrices refer to the Loewner partial order on $\Rqqsym$. The random vectors $f_1(X_i)$, $1 \le i \le n$, and $f_2(X_i,X_j)$, $1 \le i < j \le n$, are centered and uncorrelated, and
\begin{align*}
	U_n \ := \ \int_{\R^q} f \, d\hat{Q}_n \
	&= \ f_0 + 2 \int_{\R^q} f_1 \, d\hat{P}_n + M_n , \\
	U_{n,d} \ := \ \int_{\R^q} f \, d\hat{Q}_{n,d} \
	&= \ f_0 + 2 \int_{\R^q} f_1 \, d\hat{P}_n + M_{n,d} ,
\end{align*}
where
\[
	M_n \ := \ \binom{n}{2}^{-1} \sum_{1 \le i < j \le n} f_2(X_i,X_j) ,
	\quad
	M_{n,d} \ := \ (nd)^{-1} \sum_{i=1}^n \sum_{j=i+1}^{i+d} f_2(X_i,X_j) .
\]
Moreover, $\Ex(U_n) = \Ex(U_{n,d}) = f_0$, and
\begin{equation}
\label{ineq:Var.U}
	\left.\begin{array}{rll}
		n \Var(U_n)
		&= \ 4 \Gamma_1 + n \Var(M_n) &= \ 4 \Gamma_1 + 2(n-1)^{-1} \Gamma_2 \\[0.5ex]
		n \Var(U_{n,d})
		&= \ 4 \Gamma_1 + n \Var(M_{n,d}) &= \ 4 \Gamma_1 + d^{-1} \Gamma_2
	\end{array} \right\}
	\ \le \ 2 \Gamma .
\end{equation}

The final ingredient for the proof of Theorem~\ref{thm:asymptotics} is a result about the asymptotic joint distribution of $\int_{\R^q} f_1 \, d\hat{P}_n$ and $M_{n,d}$.

\begin{Proposition}
\label{prop:asymptotics.incomplete.balanced.U}
For any fixed $d \ge 1$, the random pair $\bigl( \sqrt{n} \int_{\R^q} f_1 \, d\hat{P}_n, \sqrt{nd} \, M_{n,d} \bigr)$ converges in distribution to $\NN_r(0,\Gamma_1) \otimes \NN_r(0,\Gamma_2)$.
\end{Proposition} 

\begin{proof}[\bf Proof of Proposition~\ref{prop:asymptotics.incomplete.balanced.U}]
The proof of this result uses a standard trick for sequences of $m$-dep\-en\-dent random variables, in our case, $m = d+1$. For a fixed number $k \ge d$, let
\begin{align*}
	S_n \ &:= \ n^{-1/2} \sum_{i=1}^n f_1(X_i) \ = \ \sqrt{n} \int_{\R^q} f_1 \, d\hat{P}_n , \\
	S_n^k \ &:= \ n^{-1/2} \sum_{\ell=1}^{\lfloor n/k\rfloor} Y_\ell^k
		\quad\text{with}\quad
		Y_\ell^k \ := \sum_{i=\ell k - k + 1}^{\ell k} f_1(X_i) , \\
	T_n \ &:= \ (nd)^{-1/2} \sum_{i=1}^n \sum_{j=i+1}^{i+d} f_2(X_i,X_j) \ = \ \sqrt{nd} M_{n,d} , \\
	T_n^k \ &:= \ (nd)^{-1/2} \sum_{\ell=1}^{\lfloor n/k\rfloor} Z_\ell^k
		\quad\text{with}\quad
		Z_\ell^k \ := \sum_{i=\ell k - k + 1}^{\ell k} \sum_{j=i+1}^{\min(i+d,\ell k)} f_2(X_i,X_j) .
\end{align*}
The random pairs $(Y_\ell^k,Z_\ell^k)$, $\ell \ge 1$, are independent and identically distributed with $\Ex(Y_\ell^k) = \Ex(Z_\ell^k) = 0$ and
\[
	\Var(Y_\ell^k) \ = \ k \, \Gamma_1 ,
	\quad
	\Var(Z_\ell^k) \ = \ \Bigl( k - \frac{d-1}{2} \Bigr) d \, \Gamma_2,
	\quad
	\Cov(Y_\ell^k,Z_\ell^k) \ = \ 0 .
\]
Consequently, it follows from the multivariate central limit theorem and Slutzky's lemma that
\[
	(S_n^k,T_n^k) \
	\to_{\LL}^{} \ \NN_r(0, \Gamma_1) \otimes \NN_r \Bigl( 0, \Bigl( 1 - \frac{d-1}{2k} \Bigr) \Gamma_2 \Bigr) ,
\]
and the distribution on the right hand side converges weakly to $\NN_r(0,\Gamma_1) \otimes \NN_r(0,\Gamma_2)$ as $k \to \infty$. Moreover,
\begin{align*}
	\Ex (\|S_n - S_n^k\|^2) \
	&\le \ \frac{k-1}{n} \, \mathrm{trace}(\Gamma_1) \ \to \ 0 , \\
	\Ex (\|T_n - T_n^k\|^2) \
	&\le \ \Bigl( \frac{d-1}{2k} + \frac{k-1}{nd} \Bigr) \mathrm{trace}(\Gamma_2)
		\ \to \ \frac{d-1}{2k} \, \mathrm{trace}(\Gamma_2) ,
\end{align*}
and the right hand side converges to $0$ as $k \to \infty$. This implies that $(S_n,T_n)$ converges in distribution to $\NN_r(0,\Gamma_1) \otimes \NN_r(0,\Gamma_2)$.
\end{proof}

\begin{proof}[\bf Proof of Theorem~\ref{thm:asymptotics}]
Let $\check{Q}_n$ stand for $\hat{Q}_n$, $\hat{Q}_{n,d(n)}$ with $(n-1)/2 \ge d(n) \to \infty$, or $\hat{Q}_{n,d}$ with fixed $d \ge 1$. For any bounded, continuous function $f : \R^q \to \R$,
\[
	\Ex \Bigl| \int_{\R^q} f \, d\check{Q}_n - \int_{\R^q} f \, dQ \Bigr| \ \le \ (2/n)^{1/2} \|f\|_\infty .
\]
This follows from inequality \eqref{ineq:Var.U} applied to real-valued functions. This implies that $d_{\LL}(\check{Q}_n,Q) \to_p 0$. In particular,
\[
	\bSigma(\check{Q}_n) \ = \ \bSigma(Q) + \int_{\R^q} J \, d\check{Q}_n
		+ o \Bigl( \Bigl\| \int_{\R^q} J \, d\check{Q}_n \Bigr\| \Bigr) .
\]
We may identify $\Rqqsym$ with $\R^r$, where $r = q(q+1)/2$. Then $\int_{\R^q} J \, d\check{Q}_n$ is a (complete or balanced incomplete) $U$-statistic with vector-valued kernel function, and it follows from boundedness of $J(\cdot)$ with $\int_{\R^q} J \, dQ = 0$ and the general considerations about $U$-statistics that $\int_{\R^q} J \, d\check{Q}_n = O_p(n^{-1/2})$. Consequently,
\[
	\bSigma(\check{Q}_n) \ = \ \bSigma(Q) + \int_{\R^q} J \, d\check{Q}_n + o_p(n^{-1/2}) ,
\]
so we may replace $\bSigma(\check{Q}_n) - \bSigma(Q)$ with the matrix-valued $U$-statistic $\int_{\R^q} J \, d\check{Q}_n$. But then the assertions of Theorem~\ref{thm:asymptotics} are direct consequences of the general considerations about $U$-statistics and Proposition~\ref{prop:asymptotics.incomplete.balanced.U}.
\end{proof}

For the proof of Theorem~\ref{thm:asymptotics.2} we need a variation of Proposition~\ref{prop:asymptotics.incomplete.balanced.U} for the random vectors
\[
	\tilde{M}_{n,1} \ := \ n^{-1} \sum_{i=1}^n f_2(X_{\Pi(i)}, X_{\Pi(i+1)}) ,
\]
where $\Pi$ is uniformly distributed on the set of all permutations of $\{1,2,\ldots,n\}$, independent from $(X_i)_{i=1}^n$, and $\Pi(n+1) := \Pi(1)$.

\begin{Proposition}
\label{prop:asymptotics.incomplete.randomized.U}
Let $d_{\LL}(\cdot,\cdot)$ be a metric on the space of probability distributions on $\R^r$ which metrizes weak convergence. Then,
\[
	d_{\LL} \Bigl( \LL \bigl( \sqrt{n} \, \tilde{M}_{n,1} \,
		\big| \, (X_i)_{i=1}^n \bigr), \NN_r(0, \Gamma_2) \Bigr)
	\ \to_p \ 0 .
\]
\end{Proposition} 

\begin{proof}[\bf Proof of Proposition~\ref{prop:asymptotics.incomplete.randomized.U}]
By means of the Cram\'{e}r--Wold device, it suffices to consider the case $r = 1$. Then the random variable $\sqrt{n} \, \tilde{M}_{n,1}$ can be written as $\sum_{i=1}^n A_{\Pi(i),\Pi(i+1)}$ with the random matrix
\[
	A = A(X_1,\ldots,X_n) \ := \ \bigl( n^{-1/2} 1_{[i \ne j]} f_2(X_i,X_j) \bigr)_{i,j=1}^n \in \R^{n\times n}_{\rm sym} .
\]
As explained in the appendix, there exist permutations $B = B(\cdot \,|\, \Pi)$ and $B^* = B^*(\cdot \,|\, \Pi)$ such that
\[
	\sum_{i=1}^n A_{\Pi(i),\Pi(i+1)} \ = \ \sum_{i=1}^n A_{i,B^*(i)} ,
\]
while $B$ is uniformly distributed on the set of all permutations of $\{1,\ldots,n\}$, and
\[
	\Ex \bigl( \# \bigl\{ i \in \{1,\ldots,n\} : B(i) \ne B^*(i) \bigr\} \bigr)
	\ \le \ 1 + \log(n) .
\]
Consequently,
\[
	\sqrt{n} \tilde{M}_{n,1} \ = \ \sum_{i=1}^n A_{i,B(i)} + R_n ,
\]
where $R_n := \sum_{i=1}^n (A_{i,B^*(i)} - A_{i,B(i)})$ satisfies
\[
	\Ex |R_n| \le \ 2(1 + \log(n)) n^{-1/2} \Ex |f_2(X_1,X_2)|
	\ \to \ 0 .
\]
Hence, it suffices to show that the conditional distribution of $\sum_{i=1}^n A_{i,B(i)}$, given $(X_i)_{i=1}^n$, converges weakly in probability to $\NN(0,\Gamma_2)$. Distributions of this type have been investigated by \cite{Hoeffding_1951}. It follows from Hoeffding's results and elementary inequalities presented in Section~\ref{subsec:Lindeberg} that it suffices to verify the following three properties of the random symmetric matrices $A = A(X_1,\ldots,X_n)$:
\begin{align}
\label{eq:Hoeffding.1}
	\Ex \Bigl| n^{-1} \sum_{i,j=1}^n A_{i,j}^2 - \Gamma_2 \Bigr| \
	&\to \ 0 , \\
\label{eq:Hoeffding.2}
	\Ex \Bigl( \sum_{i=1}^n \bar{A}_i^2 \Bigr) \
	&\to \ 0 , \\
\label{eq:Hoeffding.3}
	\Ex \Bigl( n^{-1} \sum_{i,j=1}^n A_{i,j}^2 \min(|A_{i,j}|,1) \Bigr) \
	&\to \ 0 ,
\end{align}
where $\bar{A}_i := n^{-1} \sum_{j=1}^n A_{i,j}$.

For an arbitrary threshold $c > 0$, we split $f_2(x,y)^2$ into the bounded function $g_c(x,y) := f_2(x,y)^2 1_{[f_2(x,y)^2 \le c]}$ and the remainder $h_c(x,y) := f_2(x,y)^2 1_{[f_2(x,y)^2 > c]}$. Then the left-hand side of \eqref{eq:Hoeffding.1} equals
\begin{align*}
	\Ex \Bigl| & n^{-2} \sum_{i,j=1}^n 1_{[i \ne j]} f_2(X_i,X_j)^2 - \Gamma_2 \Bigr| \\
	&\le \ n^{-1} \Gamma_2 + 2 \Ex h_c(X_1,X_2)
		+ n^{-2} \Ex \Bigl| \sum_{i,j=1}^n \bigl( g_c(X_i,X_j) - \Ex g_c(X_1,X_2) \bigr) \Bigr| \\
	&\le \ n^{-1} \Gamma_2 + 2 \Ex h_c(X_1,X_2)
		+ n^{-2} \Var \Bigl( \sum_{i,j=1}^n g_c(X_i,X_j) \Bigr)^{1/2} \\
	&\le \ n^{-1} \Gamma_2
		+ 2 \Ex h_c(X_1,X_2)
		+ c n^{-1/2} \\
	&\to \ 2 \Ex h_c(X_1,X_2) .
\end{align*}
The last inequality follows from the facts that
\[
	\Cov \bigl( g_c(X_i,X_j), g_c(X_k,X_\ell) \bigr) \
	\begin{cases}
		= \ 0
			& \text{if} \ \{i,j\} \cap \{k,\ell\} = \emptyset , \\
		\le \ c^2/4
			& \text{else} ,
	\end{cases}
\]
and that the number of quadruples $(i,j,k,\ell)$ with $\{i,j\} \cap \{k,\ell\} \ne \emptyset$ is smaller than $4n^3$. Since by dominated convergence, $\Ex h_c(X_1,X_2) \to 0$ as $c \to \infty$, Condition~\eqref{eq:Hoeffding.1} is satisfied.

The left-hand side of \eqref{eq:Hoeffding.2} equals $n \Ex (\bar{A}_1^2)$, and $\bar{A}_1$ is the sum of the uncorrelated, centered random variables $f_2(X_1,X_j)$, $2 \le j \le n$, times $n^{-3/2}$. Consequently,
\[
	n \Ex(\bar{A}_1^2) \ \le \ n^{-1} \Ex \bigl( f_2(X_1,X_2)^2 \bigr)
	\ \to \ 0 .
\]

Finally, the left-hand side of \eqref{eq:Hoeffding.3} is not larger than
\[
	\Ex \bigl( f_2(X_1,X_2)^2 \min\{n^{-1/2} |f_2(X_1,X_2)|, 1\} \bigr)
	\ \to \ 0
\]
by dominated convergence.
\end{proof}

\begin{proof}[\bf Proof of Theorem~\ref{thm:asymptotics.2}]
Since $(\hat{P}_n, \hat{Q}_{n,1}^{(\ell)}, M_{n,1}^{(\ell)})$ has the same distribution as $(\hat{P}_n, \hat{Q}_{n,1}, M_{n,1})$, the first assertion is a direct consequence of Theorem~\ref{thm:asymptotics} with $d = 1$. The second part is a consequence of the central limit theorem, applied to $\int_{\R^q} H_1 \, d\hat{P}_n$, and Proposition~\ref{prop:asymptotics.incomplete.randomized.U}, applied to $\sqrt{n} \, M_{n,1}^{(\ell)}$. The final statement is a consequence of the first and second part and the continuous mapping theorem.
\end{proof}

%==========================
\appendix
\section{Auxiliary results}
%==========================

%--------------------------------------------------------
\subsection{A particular coupling of random permutations}
%--------------------------------------------------------

\paragraph{Preparations.}
For an integer $n \ge 1$, let $\SS_n$ be the set of all permutations of $\langle n\rangle := \{1,2,\ldots,n\}$. A cycle in $\SS_n$ is a permutation $\sigma \in \SS_n$ such that for $m \ge 1$ pairwise different points $a_1,\ldots,a_m \in \langle n\rangle$,
\[
	a_1 \ \mapsto \ a_2 \mapsto \ \cdots \ \mapsto \ a_m \ \mapsto \ a_1 ,
\]
while $\sigma(i) = i$ for $i \in \langle n\rangle \setminus \{a_1,\ldots,a_m\}$. (In case of $m = 1$, $\sigma(i) = i$ for all $i \in \langle n\rangle$.) We write
\[
	\sigma \ = \ (a_1,\ldots,a_m)_{\rm c}
\]
for this mapping and note that it has $m$ equivalent representations
\[
	\sigma \
	= \ (a_1,\ldots,a_m)_{\rm c} \
	= \ (a_2,\ldots,a_m,a_1)_{\rm c} \
	= \ \cdots \
	= \ (a_m,a_1,\ldots,a_{m-1})_{\rm c} .
\]

Any permutation $\sigma \in \SS_n$ can be written as
\[
	\sigma \ = \ (a_{11},\ldots,a_{1m(1)})_{\rm c} \circ \cdots \circ (a_{k1},\ldots,a_{km(k)})_{\rm c} ,
\]
where the sets $\{a_{j1},\ldots,a_{jm(j)}\}$, $1 \le j \le k$, form a partition of $\langle n\rangle$. Note that the cycles $(a_{j1},\ldots,a_{jm(j)})_{\rm c}$, $1\le j\le m$, commute. This representation of $\sigma$ as a combination of cycles is unique if we require, for instance, that
\[
	a_{jm(j)} \ = \ \min\{a_{j1},\ldots,a_{jm(j)}\}
	\quad\text{for} \ 1 \le j \le k
\]
and
\[
	a_{1m(1)} \ < \ \cdots \ < \ a_{km(k)} .
\]

In what follows, let $\SS_n^*$ be the set of all permutations $\sigma \in \SS_n$ consisting of just one cycle, i.e.
\[
	\sigma \ = \ (a_1,a_2,\ldots,a_n)_{\rm c}
\]
with pairwise different numbers $a_1, a_2, \ldots, a_n \in \langle n\rangle$.

\paragraph{The coupling.}
The standardized cycle representation of $\sigma \in \SS_n$ gives rise to a particular mapping $\SS_n \ni \pi \mapsto (\sigma,\sigma^*) \in \SS_n \times \SS_n^*$ such that $\pi \mapsto \sigma$ is bijective. For fixed $\pi \in \SS_n$ and any index $i \in \langle n\rangle$ let
\[
	M_i \ := \ \langle n\rangle \setminus \{\pi(s) : 1 \le s < i\} ,
\]
i.e.\ $\langle n\rangle = M_1 \supset M_2 \supset \cdots \supset M_n = \{\pi(n)\}$, and $\# M(i) = n+1-i$. Let $1 \le t_1 < t_2 < \cdots < t_k = n$ be those indices $i$ such that $\pi(i) = \min(M_i)$. Then
\[
	\sigma \ := \ \bigl( \pi(1), \ldots, \pi(t_1) \bigr)_{\rm c}
		\circ \bigl( \pi(t_1+1),\ldots,\pi(t_2) \bigr)_{\rm c}
		\circ \cdots \circ \bigl( \pi(t_{k-1}+1), \ldots, \pi(t_k) \bigr)_{\rm c}
\]
defines a permutation of $\langle n\rangle$ with standardized cycle representation. This is essentially the construction used by \cite{Feller_1945} to investigate the number of cycles of a random permuation. Moreover,
\[
	\sigma^* \ := \ \bigl( \pi(1), \pi(2), \ldots, \pi(n) \bigr)_{\rm c}
\]
defines a permutation in $\SS_n^*$ such that
\[
	\bigl\{ i \in \langle n\rangle : \sigma(i) \ne \sigma^*(i) \bigr\}
	\ = \ \begin{cases}
		\emptyset & \text{if} \ k = 1 , \\
		\{t_1,\ldots,t_k\} & \text{if} \ k \ge 2 .
	\end{cases}
\]

Suppose that $\pi$ is a random permutation with uniform distribution on $\SS_n$. Then $\sigma$ is a random permutation with uniform distribution on $\SS_n$ too, because $\pi \mapsto \sigma$ is a bijection. Since the conditional distribution of $\pi(i)$, given $(\pi(s))_{1 \le s < i}$, is the uniform distribution on $M_i$, the random variables
\[
	Y_i \ := \ 1_{[\pi(i) = \min(M_i)]}, \quad i \in \langle n\rangle,
\]
are stochastically independent Bernoulli random variables with $\Pr(Y_i = 1) = (n+1-i)^{-1} = 1 - \Pr(Y_i = 0)$. Consequently,
\[
	\Ex \bigl( \# \bigl\{ i \in \langle n\rangle : \sigma(i) \ne \sigma^*(i) \bigr\} \bigr)
	\ \le \ \sum_{i=1}^n (n+1-i)^{-1}
	\ = \ 1 + \sum_{j=2}^n j^{-1}
	\ \le \ 1 + \log(n) ,
\]
because $j^{-1} \le \int_{j-1}^j x^{-1} \, dx = \log(j) - \log(j-1)$ for $2 \le j \le n$.

%------------------------------------------------------------------
\subsection{Some inequalities related to Lindeberg type conditions}
\label{subsec:Lindeberg}
%------------------------------------------------------------------

In connection with Gaussian approximations and Stein's method, see \cite{Stein_1986} or \cite{Barbour_Chen_2005}, the quantity
\[
	L(X) \ := \ \Ex \bigl( X^2 \min(|X|,1) \bigr)
\]
for a square-integrable random variable $X$ plays an important role. Elementary considerations show that
\[
	h(x) \ \le \ x^2 \min(|x|,1) \ \le \ \sqrt{2} \, h(x)
	\quad\text{with}\quad
	h(x) \ := \ \frac{|x|^3}{\sqrt{1 + x^2}}
\]
for arbitrary $x \in \R$. Moreover, $h : \R \to [0,\infty)$ is an even, convex function such that $h(2x) \le 8 h(x)$. Consequently, for arbitrary $x,y \in \R$, Jensen's inequality implies that
\begin{align*}
	(x + y)^2 \min (|x + y|, 1) \
	&\le \ \sqrt{2} \, \Ex h(x + y) \\
	&\le \ 2^{-1/2} \bigl( h(2x) + h(2y) \bigr) \\
	&\le \ \sqrt{32} \, \Ex h(x) + \sqrt{32} \, \Ex h(y) \\
	&\le \ 6 x^2 \min(|x|, 1) + 6 y^2 \min(|y|,1)
		\ \le \ 6 x^2 \min(|x|,1) + 6 y^2 .
\end{align*}

For a symmetric matrix $A \in \R^{n\times n}$, we define its row means $\bar{A}_i := n^{-1} \sum_{j=1}^n A_{ij}$ and its overall mean $\bar{A} := n^{-2} \sum_{i,j=1}^n A_{ij}$. Let $\tilde{A} := (A_{ij} - \bar{A}_i - \bar{A}_j + \bar{A})_{i,j=1}^n$. Then elementary calculations and the previous inequalities reveal that
\[
	0 \
	\le \ n^{-1} \sum_{i,j=1}^n A_{ij}^2 - n^{-1} \sum_{i,j=1}^n \tilde{A}_{ij}^2 \
	\le \ 2 \sum_{i=1}^n \bar{A}_i^2
\]
and
\[
	n^{-1} \sum_{i,j=1}^n \tilde{A}_{ij}^2 \min(|\tilde{A}_{ij}|,1)
	\ \le \ 6 n^{-1} \sum_{i,j=1}^n A_{ij}^2 \min(|A_{ij}|,1)
		+ 12 \sum_{i=1}^n \bar{A}_i^2 .
\]

\paragraph{Acknowledgement.}
We thank Sara Taskinen for stimulating discussions. Constructive comments of three referees are gratefully acknowledged.

%===========
% References
%===========

%\bibliographystyle{ims}
%\bibliography{refs_ApproxSymmM}
%\input{ApproxSymmM.bbl}

\end{document}